%% file: main.tex
\documentclass{article}

\RequirePackage[OT1]{fontenc}
\RequirePackage{amsmath}
\RequirePackage[numbers]{natbib}
\RequirePackage[colorlinks,citecolor=blue,urlcolor=blue]{hyperref}

\usepackage{amsmath}
\usepackage{mathtools}
\usepackage{amsthm}
\mathtoolsset{showonlyrefs}
\usepackage{amssymb}

\usepackage{amsfonts}

\title{Nonparametric Change Point Detection in Regression}
\author{Valeriy Avanesov\\
	WIAS Berlin
}

\date{\today}

\numberwithin{equation}{section}
\theoremstyle{plain}
\newtheorem{theorem}{}[section]
\newtheorem{corollary}{}[section]
\newtheorem{lemma}{}[section]
\newtheorem{definition}{}[section]
\newtheorem{assumption}{}[section]
\newtheorem{remark}{}[section]

\begin{document}
	\renewcommand{\thedefinition}{Definition \thesection.\arabic{definition}}
	\renewcommand{\thelemma}{Lemma \thesection.\arabic{lemma}}
	\renewcommand{\thetheorem}{Theorem \thesection.\arabic{theorem}}
	\renewcommand{\theassumption}{Assumption \thesection.\arabic{assumption}}
	\renewcommand{\theremark}{Remark \thesection.\arabic{remark}}
	\renewcommand{\thecorollary}{Corollary \thesection.\arabic{corollary}}

	\maketitle

	\input{commands.tex}
	\begin{abstract}
		This paper considers the prominent problem of change-point detection in regression. The study suggests a novel testing procedure featuring a fully data-driven calibration scheme. The method is essentially a black box, requiring no tuning from the practitioner. The approach is  investigated from both theoretical and practical points of view. The theoretical study demonstrates proper control of first-type error rate under $\Hnull$ and power approaching $1$ under $\Hone$. The experiments conducted on synthetic data fully support the theoretical claims. In conclusion, the method is applied to financial data, where it detects sensible change-points. Techniques for change-point localization are also suggested and investigated.
	\end{abstract}

	\input{intro.tex}
	\input{approach.tex}
	\input{theory.tex}

	\input{experiment.tex}
	\input{acknowledgement.tex}
	\appendix
	\input{validityProof.tex}
	\input{GAR.tex}
	\input{comparison.tex}
	\input{sensitivityProof.tex}
	\input{anticoncentration.tex}
	\input{freqGP.tex}

	\bibliographystyle{plain}
	\bibliography{main}

\end{document}

%% file: commands.tex
\newcommand{\g}{\mathfrak{g}}
\newcommand{\given}{\middle|}
\newcommand{\B}{\mathfrak{B}}
\newcommand{\R}{\mathbb{R}}
\newcommand{\hattau}{\hat{\tau}}
\newcommand{\N}[2]{\mathcal{N}\left(#1, #2\right)}
\newcommand{\Nat}{\mathbb{N}}
\newcommand{\E}[1]{\mathbb{E}\left[#1\right]}
\newcommand{\Var}[1]{\mathrm{Var}\left[#1\right]}
\newcommand{\norm}[1]{\left\|#1\right\|}
\newcommand{\norma}[2]{\left\|#1\right\|_{#2}}
\newcommand{\normasq}[2]{\left\|#1\right\|_{#2}^2}
\newcommand{\infnorm}[1]{\left\|#1\right\|_{\infty}}
\newcommand{\onenorm}[1]{\left|\left|#1\right|\right|_{1}}
\newcommand{\normtwo}[1]{\left\|#1\right\|_2}
\newcommand{\K}{\mathcal{K}}
\newcommand{\Fnorm}[1]{\left\|#1\right\|_F}
\newcommand{\inv}[1]{#1^{-1}}
\newcommand{\Prob}[1]{\mathbb{P} \left\{#1\right\}}
\newcommand{\xtt}{\tilde x}
\newcommand{\Krt}{K^r_n(t)}
\newcommand{\Klt}{K^l_n(t)}
\newcommand{\invKrt}{\inv{K_n^r(t)}}
\newcommand{\invKlt}{\inv{K_n^l(t)}}
\newcommand{\Probboot}[1]{\mathbb{P}^\flat \left\{#1\right\}}
\newcommand{\Hnull}{\mathbb{H}_0}
\newcommand{\Hone}{\mathbb{H}_1}
\newcommand{\fstar}{f^*}
\newcommand{\eigmin}[1]{\lambda_{\min}(#1)}
\newcommand{\eigmax}[1]{\lambda_{\max}(#1)}
\newcommand{\Deltaf}{\Delta_f}
\newcommand{\fone}{f_1^*}
\newcommand{\ftwo}{f_2^*}
\newcommand{\Bj}{B(j)}
\renewcommand{\L}{\mathfrak{L}}
\newcommand{\Dj}{D(j)}
\newcommand{\Ej}{E(j)}
\newcommand{\Z}{\mathcal{Z}}
\newcommand{\eps}{\varepsilon}
\newcommand{\epshat}{\hat{\varepsilon}}
\newcommand{\Deltasigma}{\Delta_{\sigma^2}}
\newcommand{\epsboot}{\eps^\flat}
\newcommand{\n}{\mathfrak{N}}
\newcommand{\Ant}{A_n(t)}
\newcommand{\Antboot}{A_n^\flat(t)}
\newcommand{\Tn}{\mathbb{T}_n}
\newcommand{\ylt}{y^l_n(t)}
\newcommand{\yrt}{y^r_n(t)}
\newcommand{\epsrt}{\eps^r_n(t)}
\newcommand{\frt}{f^r_n(t)}
\newcommand{\hatfrt}{\hat f^r_n(t)}
\newcommand{\hatV}{\hat{V}_n(t)}
\newcommand{\invhatV}{\inv{\hat{V}_n(t)}}
\newcommand{\yhatrt}{\hat y^r_n(t)}
\newcommand{\yt}{y_n(t)}
\newcommand{\Xlt}{X^l_n(t)}
\newcommand{\Xrt}{X^r_n(t)}
\newcommand{\znt}[1]{z_{n, #1}(t)}
\newcommand{\xnt}{x_{n, \alpha}(t)}
\newcommand{\xbar}{\bar{x}}
\newcommand{\zntboot}[1]{z_{n, #1}^{\flat}(t)}
\newcommand{\xntboot}{x_{n, \alpha}^{\flat}(t)}
\newcommand{\yhat}{\hat{y}}
\newcommand{\yboot}{y^\flat}
\newcommand{\x}{\mathrm{x}}
\newcommand{\alphastar}{\alpha^*}
\newcommand{\Iboot}{\mathcal{I}^\flat}
\newcommand{\U}{\mathcal{U}}
\newcommand{\dist}{\mathrm{dist}}
\newcommand{\relint}{\mathrm{relint}}
\newcommand{\NFdelta}{N_F(\delta)}
\newcommand{\gammap}{\gamma_p}
\newcommand{\fp}{f_{2p}}
\newcommand{\barh}{\bar{h}}
\newcommand{\xnull}{\mathring{x}}
\newcommand{\Xnull}{\mathring{X}}
\newcommand{\Znull}{\mathring{Z}}
\newcommand{\X}{\mathcal{X}}
\newcommand{\uinv}{u^{-1}}
\newcommand{\brac}[1]{\left(#1\right)}
\newcommand{\bracCurl}[1]{\left\{#1\right\}}
\newcommand{\bracSq}[1]{\left[#1\right]}
\newcommand{\nmax}{n_{+}}
\newcommand{\nmin}{n_{-}}
\newcommand{\nstar}{n_*}
\newcommand{\xnj}{x_n(j)}
\newcommand{\Znj}{Z_n(j)}
\newcommand{\xnjt}{x_n^T(j)}
\newcommand{\xnjtilde}{\tilde x_n(j)}
\newcommand{\xnjtildet}{\tilde x_n^T(j)}
\newcommand{\ynj}{y_n(j)}
\newcommand{\Xnj}{X_n(j)}
\newcommand{\ynjt}{y_n^T(j)}
\newcommand{\Sjtilde}{\tilde S_j}
\newcommand{\Sj}{S_j}
\newcommand{\Fbeta}{F_\beta}
\newcommand{\gDelta}{g_\Delta}
\newcommand{\s}{\mathfrak{s}}
\newcommand{\Zt}{\tilde{Z}}
\newcommand{\Xt}{X_n(t)}
\newcommand{\xt}{\tilde{x}}
\newcommand{\xtilde}{\tilde x}
\newcommand{\abs}[1]{\left|#1\right|}
\newcommand{\logJ}{\log \abs{J}}
\newcommand{\logJSq}{\log^2 \abs{J}}
\newcommand{\muhat}{\hat{\mu}}
\newcommand{\sigmahat}{\hat{\sigma}}
\newcommand{\sigmahatsq}{\hat{\sigma}^2}
\newcommand{\Varb}[1]{\mathrm{Var}^\flat\left[#1\right]}
\newcommand{\Ibsize}{\abs{\mathcal{I}^\flat}}
\newcommand{\e}{\mathrm{e}}
\renewcommand{\t}{\mathrm{t}}
\renewcommand{\u}{\mathrm{u}}
\newcommand{\D}{\mathcal{D}}
\newcommand{\dotprod}[2]{\langle#1,#2\rangle}
\newcommand{\RJ}{\R^{\abs{J}}}
\newcommand{\Lammax}{\Lambda_{\max}}

%% file: intro.tex
\section{Introduction}
The current study works on a problem of change point detection, which applications range from neuroimaging \cite{Bassett_Wymbs_Porter_Mucha_Carlson_Grafton_2010} to finance \cite{Serban2007,Bauwens2006,Engle1990,Mikosch2009}. In many fields practitioners have to deal with the processes subject to an abrupt unpredictable change, hence arises the need to detect and localize such changes. In the writing we refer to the former problem as {\it break detection} and the latter as {\it change-point localization}, effectively adopting the terminology suggested in \cite{avanesov2018}. The importance of the topic promotes an immense variety of considered settings and obtained results on the topic \cite{limitTheoremsCPAnalysis,aue2013,Matteson2015, Lavielle2006, aue2009, xie2013, zou2014,cho2016, jirak2015, Cho2015}.

In the current paper we focus on break detection and change point localization in regression. Typically, in a regression setting a dataset of pairs of (possibly) multivariate covariates and univariate responses is considered, while the goal is to approximate the functional dependence between the two. Here we assume, the data points are separated in time. The problem at hand is whether the functional dependence stayed the same over time and if not, when did the break take place. This setting has been attracting a plethora of attention for decades now. Most researches consider linear  \cite{Quandt, Kim1994,10.2307/2998540,BAI1999299,Bai2003,Qu2007,Martinez-Beneito2011, KIM2016106} or piece-wise constant regression \cite{bai_1997, li2016,2015arXiv150504898P}. A recent paper \cite{doi:10.1137/S0040585X97T988411} allows for a generalized linear model, leaving the proper choice of a parametric model to the practitioner. In contrast, we develop a fully non-parametric method, eliminating the need to choose a parametric family. Some papers (e.g. \cite{Quandt,BAI1999299,doi:10.1137/S0040585X97T988411}), however, rely on a fairly general framework of Likelihood-Ratio test, which we employ in our study as well.  Further, some researchers (see \cite{fryzlewicz2014} for example) propose a test statistic, yet leave the choice of the critical value to the practitioner, while we also suggest a fully data driven way to obtain them.

Contribution of our work consists in a novel break detection approach in regression which is:
\begin{itemize}
    \item fully nonparametric
    \item fully data-driven
    \item working in black-box mode: has virtually no tuning parameters
    \item capable of multiple break detection
    \item naturally suitable for change-point localization
    \item featuring formal results bounding first type error rate (from above) and power (from below)
    \item performing well on simulated and on real-world data.
\end{itemize}

Formally, we consider the pairs of deterministic multidimensional covariates $X_i \in \X$ and corresponding univariate responses $y_i \in \R$ for $i \in 1..N$, where $\X$ is a compact in $\R^p$. We wish to test a null hypothesis
\begin{equation}\label{key}
\Hnull = \left\{\forall i : y_i = \fstar(X_i) + \eps_i \right\}
\end{equation}
versus an alternative (only a single break is allowed for simplicity, Section \ref{secMult} suggests a generalization)
\begin{equation}\label{key}
\Hone = \begin{rcases}
\begin{cases}
\exists \tau, \fone \neq \ftwo : &y_i = \fone(X_i) + \eps_i \text{ if } i < \tau
\\ &y_i = \ftwo(X_i) + \eps_i \text{ otherwise}
\end{cases}
\end{rcases},
\end{equation}
where $\eps_i$ denote centered independent identically distributed noise. The functions $\fstar$, $\fone$ and $\ftwo$, mapping from the compact $\X$ to $\R$, are assumed to be unknown along with the distribution of $\eps_i$.

The approach relies on Likelihood Ratio test statistic. Assume for now, the break could happen only at the time $t$. Then it makes sense to consider $n$ data points to the left and $n$ data points to the right of $t$ and consider the ratio of likelihoods $\Ant$ of $2n$ points under a single model and under a pair of models explaining the portions of data to the left and to the right of $t$ separately. Yet the break can happen at any moment, so we consider the test statistics for all possible time moments simultaneously. Finally, in order to resolve the issue of the proper choice of the window size $n$ we suggest to consider multiple window sizes $n \in \n \subset \Nat$ at once (e. g. powers of $2$).

The paper is organized as follows. Section \ref{sectionmethod} describes the approach. Further, the approach receives a formal treatment in Section \ref{sectheor}. Finally, the behavior of the approach is empirically examined in Section \ref{secexper}.

%% file: approach.tex
\section{The approach}\label{sectionmethod}
Let us introduce some notation first. Denote the maximal and the minimal window sizes as $\nmax \coloneqq \max \n$ and $\nmin \coloneqq \min \n$. Define a set of central points for each window size $n$ as $\Tn \coloneqq \{n, n+1,..,N-n\}$. Further, for each $n\in \n$ and $t\in \Tn$ define vectors $\ylt$ composed of the responses $\{y_i\}_{i=t-n+1}^{t}$ belonging to the window to the left of $t$. Correspondingly, vectors $\yrt$ are composed of $\{y_i\}_{i=t+1}^{t+n}$. The concatenation of these two vectors is denoted as $\yt$. Also, we use $\Xlt$, $\Xrt$ and $\Xt$ to denote the tuples of covariates corresponding to $\ylt$, $\yrt$ and $\yt$ respectively. For each window size $n \in \n$ and central point $t \in \Tn$ we define the test statistic
\begin{equation}\label{key}
\Ant \coloneqq L\left(\ylt,  \Xlt\right) + L\left(\yrt, \Xrt\right) - L(\yt, \Xt),
\end{equation}
where $L$ is a likelihood function which is defined below. Intuitively, the statistic should take extremely large values when the two portions of data before and after $t$ are much better explained by a pair of distinct models than by a common one.
As we aim to construct a nonparametric approach, we define $L$ relying on a well known technique named Gaussian Process Regression \cite{Rasmussen2006}. Formally, we model the noise with a normal distribution and impose the zero-mean Gaussian Process prior with covariance function $k(\cdot, \cdot)$ on the regression function
\begin{equation}\label{key}
\begin{split}
f &\sim \mathcal{GP}\left(0,\rho k(\cdot, \cdot) \right), \\
y_j &\sim \N{f(X_j)}{\sigma^2} \text{ for } j \in 1..M,
\end{split}
\end{equation}
where $M$ is the number of response-covariate pairs under consideration and $\rho$ is a regularization parameter (see \eqref{defrho1} and \eqref{defrho2} for its choice).
Integrating $f$ out we can easily see, the joint distribution of responses $y \in \R^M$ given the covariates $X=\{X_j\}_{j=1}^M$ is modelled as a multivariate normal distribution with zero mean and covariance matrix $K\left(X\right) \in \R^{M\times M}$, such that $K\left(X\right)_{jj'}\coloneqq \rho k(X_j, X_{j'}) + \sigma^2\delta_{jj'}$, where $\delta_{jj'}$ is Kronecker delta. This observation followed by 	taking the logarithm and abolishing the non-random additive constants leads to the following definition of the likelihood $L$:
\begin{equation}\label{key}
L(y, X) \coloneqq -\frac{1}{2}y^T\inv{K(X)}y.
\end{equation}

\begin{remark}\label{remarkOnTailsWateringDown}
	The suggested approach shares its local nature with the ones presented in \cite{avanesov2018,2018arXiv180300508A, doi:10.1137/S0040585X97T988411} as they use only a portion of the dataset (of size 2n) to construct a test statistic for time $t$. Alternatively, one could use the whole dataset as in \cite{CIS-43176}, yet, this is not the best option in presence of multiple breaks.
	Consider a setting where a function $\fone$ changes to $\ftwo$ and back to $\fone$ shortly afterwards. The long tails might "water down" the test statistic. To that end a method called Wild Binary Segmentation suggests to choose multiple random continuous sub-datasets of random lengths \cite{fryzlewicz2014}. Unfortunately, this might lead to excessively long sub-datasets and significantly increase computational complexity. Our approach is free of either of these issues. Also see \ref{remarkHostileSetting} for another motivation for an approach of a local nature.
\end{remark}
\begin{remark}
Choice of covariance function $k(\cdot, \cdot)$ and $\sigma^2$ is rather important in practice. Typically, a parametric family of covariance functions $\{k_\theta(\cdot, \cdot)\}_{\theta \in \Theta}$ is considered and the optimal combination of hyper-parameters $\theta$ and $\sigma^2$ is chosen via evidence maximization (see Section 4.5.1 in  \cite{Rasmussen2006} for details).
\end{remark}
The approach being suggested rejects the $\Hnull$ if for some window size $n\in \n$ and some central point $t\in \Tn$ the test statistic $\Ant$ exceeds its corresponding critical level $\xnt$ given the significance level $\alpha$. Formally, the rejection set is
\begin{equation}\label{rejectionSet}
\left\{\exists n\in\n, t \in \Tn: \Ant > \xnt\right\}.
\end{equation}

 As the joint distribution of $\Ant$ is unknown, we mimic it with a residual bootstrap scheme in order to allow for the proper choice of the critical levels. First, let us choose some subset of indices $\Iboot \subseteq 1..N$ we use for bootstrap. We assume the response-covariate pairs $\{(y_i, X_i)\}_{i \in \Iboot}$ follow the same distribution, hence we require $\Iboot$ to be located either to the left, or to the right from $\tau$ (we presume the former without loss of generality). Given a collection of pairs $\{\left(y_i, X_i\right)\}_{i\in \Iboot}$, we construct estimates $\yhat_i$ of $\E{y_i}$ and the corresponding residuals $\epshat_i \coloneqq y_i - \yhat_i$. Now define the bootstrap counterpart of the response $y_i$ as
\begin{equation}\label{key}
\yboot_i = \yhat_i + \epsboot_i, \text{ with } \epsboot_i\coloneqq s_i \epshat_{j_i},
\end{equation}
where for all $i\in 1..N$ we draw $j_i$  independently and uniformly  with replacement from $\Iboot$ and $s_i$ are independently and uniformly drawn from $\{-1,1\}$.
At this point we can trivially define the bootstrap statistics $\Antboot$ in the same way  their real-world counterparts $\Ant$ are defined by plugging in  $\yboot_i$ instead of  $y_i$.
Next, using $\mathbb{P}^\flat$ to denote the bootstrap probability measure, we define the quantile functions for each $\x \in [0,1]$
\begin{equation}\label{key}
\zntboot{\x} \coloneqq \mathrm{inf} \left\{z: \Probboot{\Antboot > z} \le \x  \right\}.
\end{equation}
Finally, we correct the significance level $\alpha$ for multiplicity
\begin{equation}\label{multcorrection}
\alphastar \coloneqq \mathrm{sup} \left\{\x :  \Probboot{\exists n\in \n, t \in \Tn : \Antboot > \zntboot{\x}} \le \alpha \right\}
\end{equation}
and define the critical levels $\xntboot \coloneqq \zntboot{\alphastar}$.

If the method rejects $\Hnull$, one can localize the change-point as follows.
First, define the earliest central point, where $\Hnull$ is rejected
\begin{equation}
\tilde\tau^n \coloneqq \min\{t \in \Tn: \Ant > \xntboot\}.
\end{equation}
Now, if $\Ant > \xntboot$, the change point is located in the interval $[t-n, t+n)$ (up to the significance level $\alpha$). Therefore, we suggest to define the earliest detecting window
\begin{equation}
\nstar \coloneqq \arg\min_{n \in \n} (\tilde\tau^n+n)
\end{equation}
and use the following change-point location estimator
\begin{equation}\label{defhattau}
\hattau \coloneqq \arg\max_{t \in [\tilde\tau^{\nstar}-\nstar,\tilde\tau^{\nstar}+\nstar)} A_{\nmax}(t).
\end{equation}

\begin{remark}
	The estimates $\yhat_i$ may be obtained with any regression method  as long as they are consistent under $\Hnull$. As we strive for a nonparametric methodology, Gaussian Process Regression trained on $\{\left(y_i, X_i\right)\}_{i\in \Iboot}$ is suggested. The theoretical results can be trivially adapted to any kind of a consistent regressor used instead.
\end{remark}

\begin{remark}
	In practice it may be computationally difficult to obtain enough samples of the bootstrap statistics $\Antboot$ for the large number of quantiles to be simultaneously estimated. Alternatively, we suggest to choose the critical levels $\xntboot = x^\flat_{n,\alpha}$ independently of the central point $t$, effectively replacing the rejection region \eqref{rejectionSet} with
	\begin{equation}
	\left\{\exists n\in\n : \max_{t \in \Tn} \Ant > x^\flat_{n,\alpha}\right\}
	\end{equation}
	as the smaller number of quantiles can be reliably estimated based on much fewer number of the samples drawn. Clearly, this may lead to some drop of sensitivity.
\end{remark}
\begin{remark}
	The method can be easily extended for break detection in multivariate regression. In that case one can consider $A_n^l(t)$ for $l$-th component of outcome, alter the calibration scheme  accordingly and make multiplicity correction \eqref{multcorrection} also account for the dimensionality of responses (not only for the windows and break locations).
\end{remark}

\subsection{Multiple break detection}\label{secMult}
In spite of the fact that we allow for at most one break, the local nature of the test statistic $\Ant$ allows for a straightforward application of the test in presence of multiple breaks as well. Again, consider a dataset $\{(X_i,y_i)\}_{i=1}^N$ but assume $\Hone$ allows for multiple change-points $\{\tau_k\}_{k=1}^K$ ($K$ is unknown). Formally, extending the notation $\tau_0\coloneqq1$ and $\tau_{K+1}=N$,

\begin{equation}
	\Hone \coloneqq \begin{rcases}
	\begin{cases}
	&\exists \{\fstar_k\}_{k=1}^{K+1} :\fstar_k \neq \fstar_{k+1}
	\\ &y_i = \fstar_k(X_i) + \eps_i \text{ if }  \tau_{k-1} \leq i < \tau_k\\
	&\text{for all } k
	\end{cases}
	\end{rcases}.
\end{equation}
Then we estimate the location of the first change-point as
\begin{equation}
\hattau_1 \coloneqq \arg\max_{t \in [\tilde\tau^{\nstar}-\nstar,\tilde\tau^{\nstar}+\nstar)} A_{\nmax}(t).
\end{equation}
Next, the procedure is recursively called on the rest of the dataset $\{(X_i, y_i)\}_{i=\tilde\tau^{\nstar}+\nstar}^N$.

%% file: theory.tex
\section{Theoretical results}\label{sectheor}
This section is devoted to the theoretical results. Namely, Section \ref{secvalidity} presents the bootstrap validity result, claiming that the critical levels $\xntboot$ yielded by the calibration procedure are indeed chosen in accordance with the critical level $\alpha$. The sensitivity result is reported in Section \ref{secsensitivity}. It defines the minimal window width sufficient for the detection of a break and is also followed by a corollary providing change-point localization guaranties.
\input{assumption.tex}
\input{validity.tex}
\input{sensitivity.tex}

%% file: assumption.tex
\subsection{Assumptions and definitions}
In order to state the theoretical results we need to formulate some assumptions and definitions. Particularly, we rely on definition of sub-Gaussian variables and vectors.
\begin{definition}[Sub-Gaussianity]
We say a centered random variable $x$ is sub-Gaussian with $\g^2$ if
\begin{equation}
\E{\exp(sx)} \le \exp\brac{\g^2s^2/2}, ~~\forall s\in\R.
\end{equation}
We say a centered random vector $X$ is sub-Gaussian with $\g^2$ if for all unit vectors $u$ the product $\dotprod{u}{X}$ is sub-Gaussian with $\g^2$.
\end{definition}
Further, we consider two broad classes of smooth functions: Sobolev and Hölder.
\begin{definition}[Sobolev and Hölder classes]
Consider an orthonormal basis $\{\psi_j\}$ in $L_2(\R^p)$ and a function $f = \sum_j f_j \psi_j \in L_2(\R^p)$. We call it $\aleph$-smooth Sobolev if
\begin{equation}
\exists B : \sum_{j=1}^\infty j^{2\aleph}f_j^2 \le B^2
\end{equation}
and we call it $\aleph$-smooth Hölder if
\begin{equation}
\exists B : \sum_{j=1}^\infty j^{\aleph}\abs{f_j} \le B^2.
\end{equation}
\end{definition}
These properties drive the choice of the regularization parameter $\rho$. Namely, for sample size $M$ large enough we choose
\begin{equation}\label{defrho1}
    \rho = \frac{B^2}{\log M}
\end{equation}
if the function is Sobolev and
\begin{equation}\label{defrho2}
    \rho = \brac{\frac{B^2}{\log M}}^{2\aleph / (2\aleph+1)} \brac{\frac{1}{M}}^{1/(2\aleph+1)}
\end{equation}
if the function is $\aleph$-Hölder.

Throughout the paper we use a variety of norms. We use $\norm{\cdot}$ to denote Euclidean norm of a vector or a spectral norm of a matrix. Further, $\infnorm{\cdot}$ refers to sup-norm for both vectors and matrices (the maximal absolute value of an element), as well as functions (the maximal absolute value of an element of its image), while $\Fnorm{\cdot}$ stands for Frobenius norm of a matrix.

The result \ref{gpconsistency} (by \cite{Yang2017}) we rely upon imposes the following two assumptions.
\begin{assumption}\label{keras1}
    Let there exist $C_\psi$ and $L_\psi$ s.t. for eigenfunctions $\{\psi_j(\cdot)\}_{j=1}^\infty$ of covariance function $k(\cdot, \cdot)$
    \begin{equation}
    \max_j\infnorm{\psi_j} \le C_\psi
    \end{equation}
    and for all $t,s \in \R^p$
    \begin{equation}
    \abs{\psi_j(t) - \psi_j(s)} \le jL_\psi \norm{t-s}.
    \end{equation}
\end{assumption}

\begin{assumption}\label{keras2}
Let for the eigenvalues $\{\mu_j\}_{j=1}^\infty$ of covariance function $k(\cdot, \cdot)$ exist positive $c$ and $C$ s.t. $c j^{-2\aleph} \le \mu_j \le C j^{-2\aleph}$ for $\aleph > 1/2$.
\end{assumption}
Mat\'ern kernel with smoothness index $\aleph-1/2$ satisfy these assumptions. In \cite{Yang2017} the authors claim, their results also hold for kernels with non-polynomially decaying eigenvalues, like RBF and polynomial kernels. And as long as we do not use these assumptions in our proofs directly, so do ours.

Finally, we introduce the assumptions required by our machinery.
\begin{assumption}\label{strangeAssumption}
    Let $\inv{\tilde K_n(t)}$ have the same elements as $\inv{K(\Xt)}$ with exception for the diagonal and $\mathrm{diag} \inv{\tilde K_n(t)} = 0$. Assume, exists a positive $\gamma$ s.t. for all $t \in \Tn$ for $n\rightarrow \infty$
    \begin{equation}\label{strangeEq}
        \infnorm{\inv{\tilde K_n(t)} \E{\yt}} = O(n^\gamma).
    \end{equation}
\end{assumption}
It would be natural to expect $\inv{K(\Xt)}$ in \eqref{strangeEq} instead of $\inv{\tilde K_n(t)}$, e.g.
\begin{equation}\label{strangeEqPrime}
    \infnorm{\inv{ K(\Xt)} \E{\yt}} = O(n^\gamma).
\end{equation}
 On the one hand, if the design $\{X_i\}_{i=1}^N$ is regular, (e.g. a uniform grid), \eqref{strangeEqPrime} implies \eqref{strangeEq}, yet in general, particularly \eqref{strangeEq} is the desired assumption. We prove the bootstrap validity result (\ref{validityTheorem}) using our Gaussian approximation \ref{garlemma}. There we have to treat the diagonal and off-diagonal elements of the quadratic forms separately. This is reminiscent of the results in \cite{goetze} where they study an asymptotic distribution of a single quadratic form (we, in contrast, work with a joint distribution of numerous quadratic forms).

\begin{assumption}\label{assboundedEigenvalues}
    Let there exist a positive constant  $C$ independent of $n$ s.t. $\forall t$
    \begin{equation}
        \norm{K(\Xt)} < C.
    \end{equation}
\end{assumption}
Informally, \ref{strangeAssumption} does not let the GP prior be too unrealistic, while \ref{assboundedEigenvalues} prohibits concentrations of measurements in a local area. Neither would we like \ref{assboundedEigenvalues} violated looking from a practical perspective, as it ensures $K(\Xt)$ being well-conditioned.

%% file: validity.tex
\subsection{Bootstrap validity}\label{secvalidity}
In this section we demonstrate closeness of measures $\mathbb{P}$ and $\mathbb{P}^\flat$ in some sense which is a theoretical justification of our choice of the calibration scheme.

\begin{theorem}\label{validityTheorem}
	Let $\Hnull$, \ref{keras1}, \ref{keras2} hold, $\eps_i$ be sub-Gaussian with $\g^2$. Let $\fstar$ be $\aleph$-smooth Sobolev and $\kappa \coloneqq \brac{\aleph-1/2}/(2\aleph)$  or $\aleph$-smooth Hölder and $\kappa\coloneqq \aleph/(2\aleph+1)$.
	Let $\nmin$, $\nmax$, $\abs{\n}$ and $N$ grow. Also assume for some positive $\gamma$ and $\delta$
	\begin{equation}\label{windowKIllsLog}
	\frac{\log^{15}\brac{\abs{\n} N}}{\nmin^{1-6\gamma}} = o(1),
	\end{equation}
	\begin{equation}\label{ibootsizeAssumption}
	\brac{\frac{\log \Ibsize}{\Ibsize}}^{\kappa}\nmax^{1/2+\delta+\gamma} = o(1)
	\end{equation}
	and finally, let \ref{strangeAssumption} hold for $\gamma$.
	Then on a set of arbitrarily high probability
	\begin{equation}\label{ibsizekills}
	\begin{split}
	\sup_{c_n(t)} &\abs{\Prob{\forall n \in \n, t \in \Tn : A_n(t) < c_n(t)} -\right.\\&\left. \Probboot{\forall n \in \n, t \in \Tn :  A^\flat_n(t) <  c_n(t)}} = o(1).
	\end{split}
	\end{equation}
\end{theorem}
Proof of the theorem is given in Section \ref{ProofValiditySec}. The strategy of the proof is typical for bootstrap validity results. First, we approximate the joint distribution of the test statistics $\{\Ant\}_{n\in\n, t\in\Tn}$ with a distribution of some function of a high-dimensional Gaussian vector. This step is handled with our Gaussian approximation result \ref{assGarLemma}. Next, the same is done for their bootstrap counterparts $\{\Antboot\}_{n\in\n, t\in\Tn}$ using a different Gaussian vector. Finally, we build the bridge between the two approximating distributions using the fact that the mean and variance of these Gaussian vectors are close to each other (see \ref{lemmacomparison}).
The assumptions \eqref{windowKIllsLog} and \eqref{ibootsizeAssumption} enforce negligibility of the remainder terms involved in \ref{assGarLemma} and \ref{lemmacomparison} respectively.
In turn, the Gaussian approximation result (\ref{assGarLemma}) is obtained using a novel, significantly tailored version of Lindeberg principle \cite{paulauskas2012approximation, rotar, chatterjee2005, chatterjee2006}. The proof of Gaussian comparison result (\ref{lemmacomparison}) is inspired by the technique used in \cite{chernComp}. We use Slepian smart interpolant too, yet applying it in a non-trivial way. We believe, \ref{assGarLemma} can also be proven via Slepian smart interpolant instead of Lindeberg principle, which might yield slightly better convergence rate. We leave this for the future research.

%% file: sensitivity.tex
\subsection{Sensitivity result}\label{secsensitivity}
Consider a setting under $\Hone$. For simplicity, assume there is a single change point at $\tau$. In order for the break to be detectable we have to impose some discrepancy condition on $\fone$ and $\ftwo$. Moreover, in order to guarantee detection we have to require the choice of covariates $X_i$ to make this discrepancy observed. Keeping that in mind we define the {\it observed break extent}
\begin{equation}\label{defBreakExtent}
\B_n^2 \coloneqq \frac{1}{n} \sum_{i=\tau}^{\tau+n-1} \brac{\fone(X_{i}) - \ftwo(X_{i})}^2.
\end{equation}

\begin{theorem} \label{sensitivityTheorem}
	Let the setting described above hold, $\eps_i$ be sub-Gaussian with $\g^2$. Let $\fstar$, $\fone$, $\ftwo$ be $\aleph$-smooth Sobolev and $\kappa \coloneqq \brac{\aleph-1/2}/(2\aleph)$  or $\aleph$-smooth Hölder and $\kappa\coloneqq \aleph/(2\aleph+1)$. Also let $\nstar \in \n$, $\nstar,N \rightarrow +\infty$ and $\B_{\nstar} \rightarrow 0$. Also impose \ref{keras1}, \ref{keras2}, \ref{strangeAssumption} (for $t < \tau$), \ref{assboundedEigenvalues}, \eqref{windowKIllsLog}, \eqref{ibootsizeAssumption} and
	\begin{equation}\label{deltaRate}
	\B_{\nstar}^{-1}  \brac{\frac{\log \nstar}{ \nstar}}^{\kappa}= o(1).
	\end{equation}
	Then
	\begin{equation}\label{key}
	\Prob{\Hnull \text{ is rejected}} \rightarrow 1.
	\end{equation}
\end{theorem}
We defer proof to Appendix \ref{sensProofsec}. It is fairly straightforward. First, we bound the test statistics $A_n(\tau)$ with high probability, next we use \ref{validityTheorem} to also bound the critical levels $x_{n,\alpha}(\tau)$ and finally, we bound the test statistic $A_n(\tau)$ from below and make sure it exceeds the critical level. The assumption \eqref{deltaRate} essentially requires the observed break extent to exceed the precision of Gaussian Process Regression predictor.
\begin{remark}\label{remarkHostileSetting}
	The sensitivity result gives rise to another motivation behind simultaneous consideration of wider and narrower windows (and also it is another argument for local statistics in the first place, also see \ref{remarkOnTailsWateringDown}). Consider a hostile setting, where the values of functions $\fone$ and $\ftwo$ coincide for most of the arguments. For instance, let $\B'\coloneqq \abs{\fone(X_\tau) -\ftwo(X_\tau)}$ and let $\fone(X_i) = \ftwo(X_i)$ for all $i > \tau$. Then by definition $\B_n = \B'/n$ and hence the assumption \eqref{deltaRate} implies
	\begin{equation}
		\B'^{-1}\nstar^{1-\kappa}\log^\kappa\nstar = o(1).
	\end{equation}
	Clearly, a narrower window detects a smaller break of such a kind.
\end{remark}
\begin{remark}\label{remarkMult}
	In the setting allowing for multiple change-points (see Section \ref{secMult}) assumption \eqref{deltaRate} dictates the requirement for the minimal distance $\Delta_\tau \coloneqq \min_{k,k':k\neq k'} \abs{\tau_k - \tau_{k'}} $ between two consecutive change-points as $\Delta_\tau \ge 2\nstar + \abs{\Iboot}$ which is sufficient for detection of all the change-points with probability approaching $1$.
\end{remark}
Finally, we formulate a trivial corollary providing change-point localization guaranties.
\begin{corollary}
	Under the assumptions of  \ref{sensitivityTheorem}
	\begin{equation}
		\Prob{\abs{\tilde\tau - \tau} \le \nstar} \gtrsim 1-\alpha.
	\end{equation}
\end{corollary}

%% file: experiment.tex
\section{Empirical Study}\label{secexper}
In this section we report the results of our experiments\footnote{The code is available at \url{github.com/akopich/gpcd}}. Section \ref{secsynth} presents the findings of the simulation study supporting the bootstrap validity and sensitivity results, as well as empirically justifying the simultaneous use of multiple windows and the change-point location estimator \ref{defhattau}. In Section \ref{seqnasdaq} we successfully apply the method to detect change-points in daily quotes of NASDAQ Composite index.
\subsection{Experiment on synthetic data} \label{secsynth}
We consider functions $\fstar(x) = \fone(x) = \sin(x)$ and $\ftwo(x) = \sin(x + \phi)$ for various choices of $\phi$. Univariate covariates $\{x_i\}_{i=1}^{800}$ are shuffled $800$ equidistant points between $0$ and $\pi$. Under $\Hnull$ the responses are sampled independently as $y_i \sim \N{\fstar(x_i)}{0.1^2}$. Under $\Hone$ we choose the change-point location  $\tau = 700$ and sample $y_i \sim \N{\fone(x_i)}{0.1^2}$ for $i < \tau$ and  $y_i \sim \N{\ftwo(x_i)}{0.1^2}$ for $i \ge \tau$.
For our experiments we consider $\phi \in \{\pi/2,\pi/5, \pi/10, \pi/20, \pi/40\}$ and report the corresponding observed break extent $\B_n$ (defined by \eqref{defBreakExtent}). In all the experiments $\Iboot = \{1, 2, .., 500\}$, the confidence level $\alpha$ was chosen to be $0.01$. We choose RBF kernel family
\begin{equation}\label{defrbf}
k_\theta(x_1, x_2) = \theta_1^2\exp \brac{-\frac{|x_1-x_2|^2}{\theta_2^2}}
\end{equation}
and choose optimal parameters $\theta$ and $\sigma^2$ via evidence maximization using $\{x_i\}_{i \in \Iboot}$.

The suggested approach has demonstrated proper control of the first type error rate in all the configurations we consider, keeping it below $0.015$.

The power the test exhibits is shown on Figure \ref{powerOnWindow}. As expected, larger window size $n$ and larger observed break extent $\B_n$ correspond to higher power. At the same time, the Figure \ref{rmseOnWindow} summarizes root mean squared errors of the estimator $\hattau$ (defined by \eqref{defhattau}). The estimator proves itself to be reliable when the power of the test is high. Generally,  wider windows and larger observable break extent lead to higher accuracy of $\hattau$.

Further, in order to investigate the behavior of the method in a multiscale regime ($\abs{\n}>1$) we use several choices of $\n$ for a single $\phi=\pi/10$. Results, reported in the Table \ref{multiscaletable}, exhibit a significant decrease in the average width of the narrowest detecting window $\nstar$ and hence an improvement in change-point localization thanks to simultaneous use of wider and narrower windows. This should be highly beneficial in presence of multiple change points, as it allows for smaller distance $\Delta_\tau$ between them (see Section \ref{secMult} and \ref{remarkMult}).

\begin{figure}
  \includegraphics[width=\textwidth]{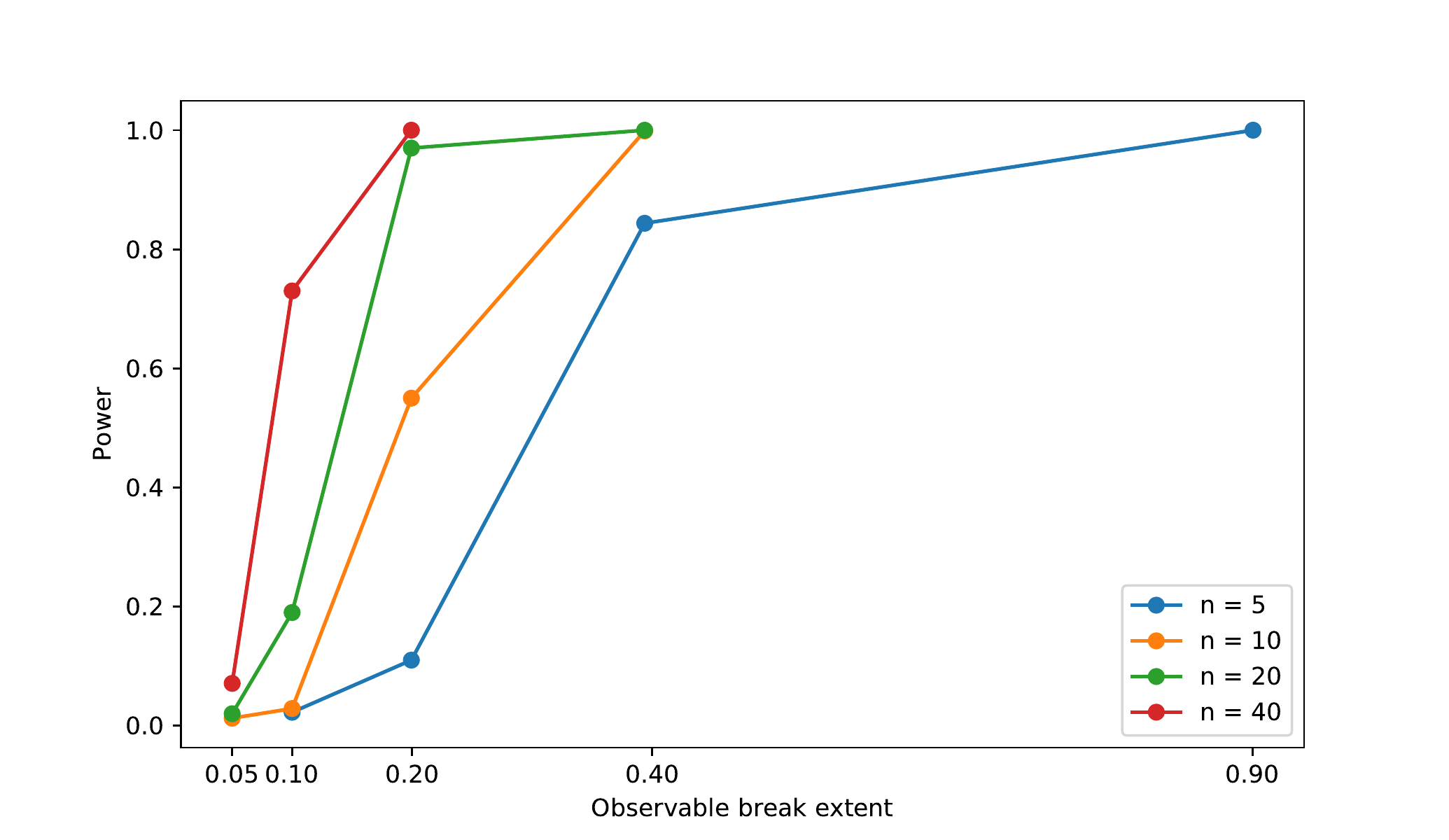}
  \caption{The plot demonstrates the dependence of the power of the test on the observable break extent $\B_{n}$ (see  \eqref{defBreakExtent}) for multiple window sizes $n$.}
  \label{powerOnWindow}
\end{figure}
\begin{figure}
  \includegraphics[width=\textwidth]{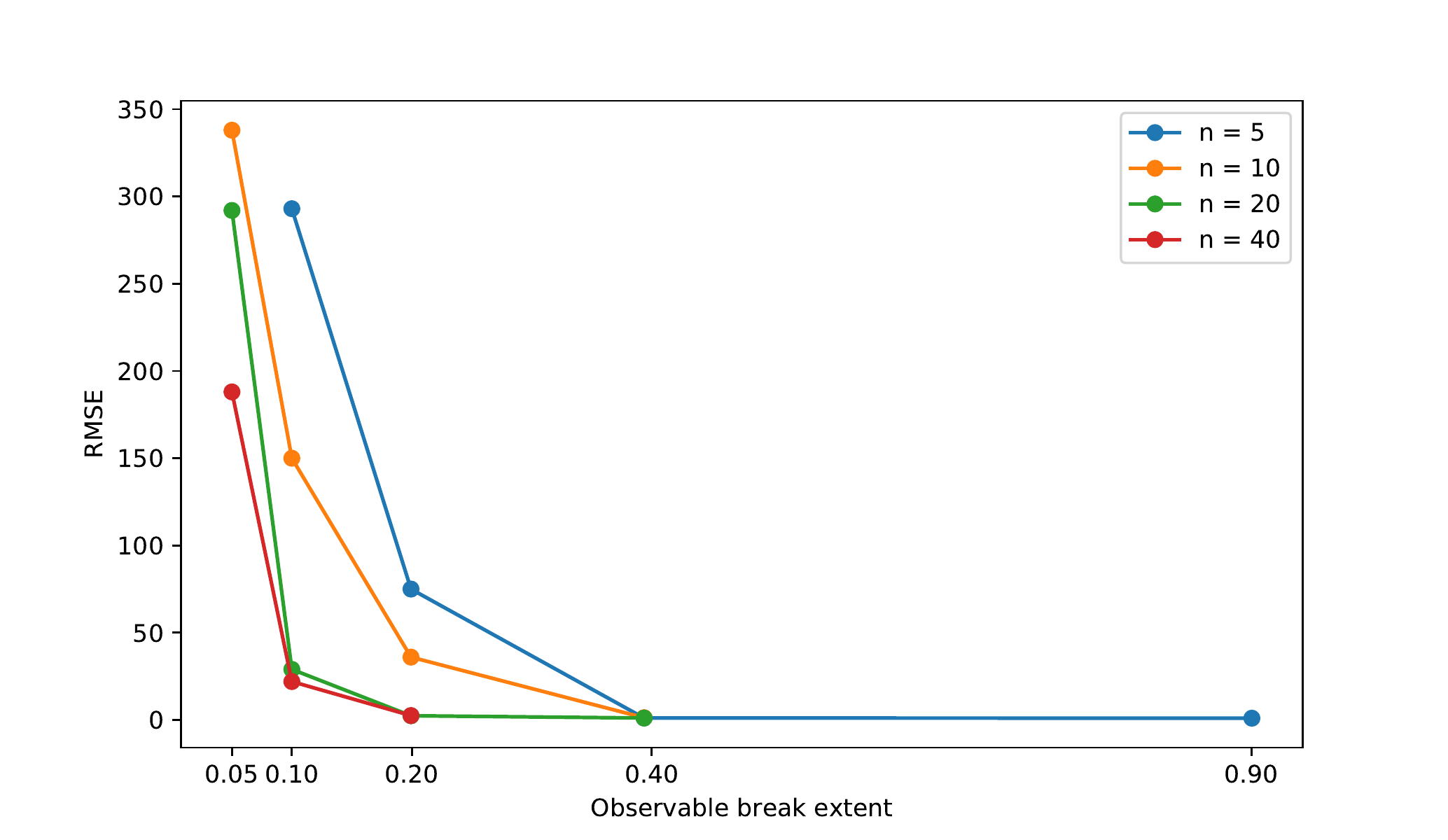}
  \caption{The plot demonstrates the dependence of the root mean squared error of change point localization on the observable break extent $\B_n$ (see \eqref{defBreakExtent}) for multiple window sizes $n$.}
  \label{rmseOnWindow}
\end{figure}

\begin{table}\caption{\label{multiscaletable} This table demonstrates average width of the narrowest detecting window $\nstar$ (from \ref{sensitivityTheorem}) may be reduced employing multiple window sizes at once without noticeable loss of power.}
\centering
\begin{tabular}{|c|c|c|}
\hline
$\n$                      & Power  & $\nstar$ \\ \hline
$\{40\}$            & 1.0   & 40.0   \\ \hline
$\{40, 20\}$        & 1.0   & 20.5  \\ \hline
$\{40, 20, 10\}$    & 1.0  & 15.7  \\ \hline
$\{40, 20, 10, 5\}$    & 1.0  & 15.9  \\ \hline
\end{tabular}
\end{table}

\subsection{Real-world dataset experiment}\label{seqnasdaq}
The prices of stock indexes are known to be subject to abrupt breaks \cite{NIPS2011_4286, pmlr-v22-stimberg12}. We consider a series $X_t$ of closing daily prices of NASDAQ Composite index. The dataset spans from February 1990 until February 2019. We suggest to model the process using the following Stochastic Differential equation
\begin{equation}
\frac{dX_t}{X_t} = f(X_t) dt + \sigma dW_t\text{, } \sigma>0,
\end{equation}
where $W_t$ denotes a Wiener process. Now we wish to test the dataset for the presence of breaks. In order to do so we employ the Euler–Maruyama method, effectively boiling the problem down to a regression problem with univariate covariates $x_t\coloneqq X_t$ and the corresponding responses $y_t \coloneqq \frac{X_{t+1}-X_t}{X_t}$. Further we apply the scheme suggested in Section \ref{secMult} with $\alpha=0.01$, $\n=\{20\}$, $\Iboot=\{1..300\}$ and the kernel family \eqref{defrbf}. The method detects three breaks and all of them may be related to the known events. Namely, computer virus CIH has activated itself and attacked Windows 9x in August 1998, burst of the dot-com bubble and 2008 financial crisis.

%% file: acknowledgement.tex
\section*{Acknowledgements}
The research of ``Project Approximative Bayesian inference and model selection for stochastic differential equations (SDEs)'' has been partially funded by Deutsche Forschungsgemeinschaft (DFG) through grant CRC 1294 ``Data Assimilation'', ``Project Approximative Bayesian inference and model selection for stochastic differential equations (SDEs)''.

Further, we would like to thank Vladimir Spokoiny, Alexandra Carpentier and Evgeniya Sokolova for the discussions and/or proofreading which have greatly improved the manuscript.

%% file: validityProof.tex
\section{Proof of the bootstrap validity result}\label{ProofValiditySec}
\begin{proof}[Proof of \ref{validityTheorem}]
	Apply \ref{assGarLemma} to $\Ant$ and $\Antboot$, next apply \ref{lemmacomparison} and via triangle inequality obtain on a set of probability at least $1-2\exp(-\u^2)$
	\begin{equation}\label{key}
	\sup_{c_t} \abs{\Prob{A_n(t) < c_t} - \Probboot{A^\flat_n(t) < c_t}}  \le 2R_A + R_C,
	\end{equation}
	where $R_A$ and $R_C$ come from \ref{assGarLemma} and \ref{lemmacomparison} respectively.
	Now observe
	\begin{equation}
		\abs{J} = \sum_{n \in \n } \abs{\Tn} \le \abs{\n}N
	\end{equation}
	and using \eqref{windowKIllsLog} conclude $R_A = o(1)$. Clearly, the ratio entering the definition of $R_C$ is bounded $\sqrt{n}/\s = O(1)$ (in the same way as in the proof of \ref{assGarLemma}). Next we use \ref{gpconsistency} and obtain on a set of probability at least $1-\Ibsize^{-10}$
	\begin{equation}\label{key}
	\infnorm{\E{y} - \yhat} \le O(\Delta_f),
	\end{equation}
	where
	\begin{equation}
		\Delta_f \coloneqq \brac{\frac{\log \Ibsize}{\Ibsize}}^{\kappa}.
	\end{equation}
	Now observe that the following holds for $\Delta_{\mu}$ and $\Delta_{\Sigma}$ involved in \ref{lemmacomparison} by construction of $Z$ and $\tilde Z$ (coming from the gaussian approximation and defined by \eqref{Ztdef})
	\begin{equation}\label{key}
	\Delta_{\mu} \le \brac{\Delta_f+\infnorm{\fstar}}\Delta_f,
	\end{equation}
	\begin{equation}\label{key}
	\Delta_{\Sigma} \le 4{\abs{\Var{\eps_1} - \Var{\epsboot_1}}}\infnorm{\fstar}^2 + \abs{\Var{\eps_1} - \Var{\epsboot_1}}^2.
	\end{equation}
	Further, \ref{sigmabootisclose} yields the bound   $\abs{\Var{\eps_1} - \Var{\epsboot_1}}=O(\Delta_f^2)$.
	Assumption \eqref{windowKIllsLog} implies $\gamma < 1/6$. Then \eqref{ibootsizeAssumption} it turn implies
	\begin{equation}
	\brac{\frac{\log \Ibsize}{\Ibsize}}^{\kappa}\nmax^{\delta/2+\gamma} = o(1).
	\end{equation}
	 Finally, choose $\Delta = n^{-\delta/2}$ (involved in the definition of $R_C$, see \ref{lemmacomparison}), recall assumption \eqref{ibootsizeAssumption} and conclude $R_C = o(1)$.
\end{proof}

%% file: GAR.tex
\section{Gaussian Approximation}\label{sectionGAR}

Consider a random vector $x \in \R^N$ of independent components centered at $\mu = \E{x}$. Introduce $x_n(j)$ for even $n = 2m \in \mathbb{N}$ and $j \in J \coloneqq \{m, m+1,..,N-m\}$ denoting a vector composed of $\{x_i\}_{i=j-m+1}^{j+m}$. Also, assume $\abs{J}$ symmetric matrices $B(j) \in \R^{n\times n}$ are given and define a map $S: \R^N \rightarrow \R^{\abs{J}}$ s. t. $S(x)_j \coloneqq \frac{1}{\sqrt{n}}\dotprod{\xnj}{ B(j) \xnj}$.

Two ingredients of paramount importance are soft-max function $\Fbeta : \R^{\abs{J}} \rightarrow \R$
\begin{equation}\label{key}
\Fbeta(w) \coloneqq \log \brac{\sum_{j \in J} \exp\brac{\beta w_j}} \text{ for } \beta\in\R
\end{equation}
and a smooth indicator function $g_\Delta$ with three bounded derivatives s.t. $\abs{x} \ge \Delta \Rightarrow g_\Delta(x) = 1[x > 0]$. Also let $g \coloneqq g_1$ and $g(x/\Delta) = \gDelta(x)$. An example of such function along with bounds for its derivatives is provided in \cite{doi:10.1137/S0040585X97T988411}.

Consider the following decomposition of matrices $\Bj$ into diagonal matrices and matrices with zeroes down their diagonals
\begin{equation}\label{diagoffdiagdecomposition}
\Bj = \Ej + diag(\Dj)\text{, where } \Ej_{kk} = 0 ~\forall k.
\end{equation}
Further, consider a vector $X\in \R^N$ s.t. $x_i^2 = X_i$ for all $i = 1..N$. And introduce notation $\Xnj$ similar to $\xnj$. Now consider a vector $Z$ denoting vectors $x$ and $X$ stacked. Clearly, there is a map $Q: \R^{2N}\rightarrow \RJ$ s.t.:
\begin{equation}\label{Qdef}
S(x)_j = Q(Z)_j \coloneqq Q(x, X)_j \coloneqq \frac{1}{\sqrt{n}}\dotprod{\xnj}{ E(j) \xnj}  + \frac{1}{\sqrt{n}}\dotprod{D(j) }{ \Xnj}
\end{equation}
for all $j=1..\abs{J}$. Also define an independent vector \begin{equation}\label{Ztdef}
\Zt \sim \N{\E{Z}}{\Var{Z}}
\end{equation} and denote the first half of the vector as $\xt$ and the second as $\tilde X$.

Our proof employs a novel version of the Lindeberg principle \cite{paulauskas2012approximation, rotar, chatterjee2005, chatterjee2006} tuned for the problem at hand. Typically, Lindeberg principle suggests to "replace" random variables with their Gaussian counterparts one by one. Here we have to "replace" each $n$-th component of $x$ along with the component of $X$ being its square starting with the $1$-st one, repeat starting with the $2$-nd one and so on repeating the procedure $n$ times.
Namely, in the first step we "replace" components with indexes $1$, $n+1$, $2n+1$ and so on. On the second step we "replace" components with indexes $2$, $n+2$, $2n+2$ and so on. And further in the same manner. Or more formally, consider a sequence of vectors $x^i \in \R^N$ for $i = 0..n$ s. t. $x^0=x$ and  $\forall i>0: x^{i}_{kn+i} = \tilde x_{kn+i}$ for all $k \in \{0, 1, 2, ..,\lceil N/n\rceil- 1\}$ and $x^{i}_{j} = x^{i-1}_{j}$ for $j$ s.t. $\nexists k \in \{0, 1, 2, ..,\lceil N/n\rceil-1 \} : kn+i=j$. Denote the indexes of components which were replaced at step $i$ as $I_i$. Also define a vector $\xnull^i$ s.t. $\xnull^i_j = 0$ for $j \in I_i$ and $\xnull^i_j = x^i_j$ for the rest of $j$. Define sequence of $X^i$ and $\mathring{X}^i$ in a similar way. Finally, let $Z^i$ denote the vectors $x^i$ and $X^i$ stacked together and $\Znull_i$ denote stacked vectors $\xnull^i$ and $\mathring{X}^i$. Note, $Z^n = \Zt$.

\begin{lemma}\label{boundPhi}
	Choose $i=1..n$. Consider a function $\phi : \R^N\times\R^N \rightarrow \R$ defined as
	\begin{equation}\label{key}
	\phi(a,b) \coloneqq \Fbeta\brac{Q\brac{\xnull^i+a, \Xnull^i+b}},
	\end{equation}
	where $j\notin I_i \Rightarrow a_j = 0,~b_j = 0$ and $Q(\cdot)$ is defined by \eqref{Qdef}. Further, using decomposition \eqref{diagoffdiagdecomposition} assume for some positive $L$:
	\begin{equation}\label{Lbound}
	\max_j \infnorm{\Ej(\xnull^i_n(j)+a_n(j))}< L
	\end{equation}
	and denote
	\begin{equation}\label{Ldef}
		\L \coloneqq \max\left\{2L, \max_{j \in J} \infnorm{D(j)}\right\}.
	\end{equation}
	Then
	\begin{equation}\label{key}
	\abs{\sum_{j\in I_i} \brac{\partial_{a_j}+\partial_{b_j}}\phi(a,b)} \le \varphi'\coloneqq \frac{2\L}{\sqrt{n}},
	\end{equation}
	\begin{equation}\label{key}
	\abs{\sum_{j,j'\in I_i} \brac{\partial^2_{a_ja_{j'}}+\partial^2_{b_jb_{j'}}}\phi(a,b)} \le \varphi ''\coloneqq \frac{4\beta\L^2}{n}
	\end{equation}
	\begin{equation}\label{key}
	\abs{\sum_{j,j',j''\in I_i} \brac{\partial^3_{a_ja_{j'}a_{j''}}+\partial^3_{b_jb_{j'}b_{j''}}}\phi(a,b)} \le \varphi'''\coloneqq \frac{12\beta^2\L^3}{n^{3/2}}.
	\end{equation}
\end{lemma}
\begin{proof}
	Proof of this result consists in direct differentiation followed by application of Lemma A.2 from \cite{chernozhukov2013}  providing bounds for the first three derivatives of soft-max function.
\end{proof}

\begin{lemma} \label{lemmatelescope} Let assumptions of \ref{boundPhi} hold. Then for an independent Gaussian vector $\Zt$ (defined by \eqref{Ztdef})
	\begin{equation}\label{key}
	\abs{\E{\gDelta\circ \Fbeta \circ S(x)} - \E{\gDelta\circ \Fbeta \circ Q(\Zt)}} \le \frac{n}{6}  \Z m_3,
	\end{equation}
	where $m_3$ is the sum of the maximal third centered absolute moments of $x$ and $\Zt$, while $\Z$ is defined in \ref{onetermdeeper} and $Q(\cdot)$ is defined by \eqref{Qdef}.
\end{lemma}
\begin{proof}
	Clearly, for $f \coloneqq \gDelta \circ \Fbeta \circ Q$,
	\begin{equation}\label{key}
	f(Z) - f(\Zt) = \sum_{i=1}^n f\brac{Z^{i-1}}  - f\brac{Z^i}
	\end{equation}
	and hence
	\begin{equation}\label{telescope}
	\abs{\E{f(Z) - \E{f(\tilde Z)}}} \le \sum_{i=1}^n \abs{\E{f\brac{Z^{i-1}}} - \E{f\brac{Z^{i}} }}.
	\end{equation}
	The rest of the proof consists in bounding an arbitrary summand on the right hand side. In order to do so we use Taylor expansion of second degree for ${f\brac{Z^{i-1}}}$ and ${f\brac{Z^{i}}}$ around $\E{Z}$ with Lagrange remainder. Given equality of the first two moments of $Z$ and $\Zt$, we conclude, the first two terms cancel out. Hence, using \ref{onetermdeeper} we immediately obtain
	\begin{equation}\label{oneTermTelescope}
	\abs{\E{f\brac{Z^{i-1}}} - \E{f\brac{Z^{i}} }} \le \frac{1}{6}  \Z m_3.
	\end{equation}
	Combination of \eqref{telescope}  and \eqref{oneTermTelescope} establishes the claim.
\end{proof}

\begin{lemma}\label{garlemma} Let assumptions of \ref{boundPhi} hold.  Then
\begin{equation}\label{key}
\begin{split}
\sup_c&\abs{\Prob{S(x) \le c} - \Prob{Q(\Zt) \le c}} \le R_A \\ &\coloneqq  84\log^{3/2}\abs{J} \brac{1 + \frac{\sqrt{n}}{\s}}\frac{\L^{3/4}}{n^{1/8}},
\end{split}
\end{equation}
where $\s$ comes from \ref{anticoncentration} and $Q(\cdot)$ is defined by \eqref{Qdef}.
\end{lemma}
\begin{proof}
	Choose  $\Delta = \logJ/\beta$. Then for an arbitrary constant vector $c \in \RJ$
	\begin{equation}\label{key}
	\begin{split}
	\Prob{S(x) \le c } &\le \E{\gDelta(\Fbeta(S(x)-c+\Delta))} \\
	& \le \E{\gDelta(\Fbeta(Q(\Zt)-c+\Delta))} + \frac{n}{6} \Z m_3 \\
	& \le \Prob{Q(\Zt)\le c-2\Delta} + \frac{n}{6} \Z m_3\\
	& \le \Prob{Q(\Zt)\le c} + 12 \Delta\frac{\sqrt{n}}{\s}\brac{\sqrt{4\logJ}+2} + \frac{n}{6} \Z m_3 \\
	 \le  \Prob{Q(\Zt)\le c}& +24 \Delta\frac{\sqrt{n}}{\s}\sqrt{4\logJ} + \frac{768\L^3\logJSq}{\Delta^3\sqrt{n}} .
	\end{split}
	\end{equation}
	Here he have consequently used \ref{boundsForSmoothing}, \ref{lemmatelescope}, \ref{boundsForSmoothing} again and \ref{anticoncentration} (which also defines $\s$). The last step uses that $\log 3>1$. Now we choose
 	\begin{equation}\label{key}
 	\Delta = 8 \brac{\frac{144\log^{3/2}\abs{J}\L}{\sqrt{n}}}^{1/4}
 	\end{equation} and obtain
	\begin{equation}\label{key}
	\begin{split}
	\Prob{S(x) \le c } &\le  \Prob{Q(\Zt)\le c} + \brac{84\log^{3/2}\abs{J}\frac{\sqrt{n}}{\s}  + 13 \log^{7/8}\abs{J}}\frac{\L^{3/4}}{n^{1/8}}\\
	& \le  \Prob{Q(\Zt)\le c} + 84\log^{3/2}\abs{J} \brac{1 + \frac{\sqrt{n}}{\s}}\frac{\L^{3/4}}{n^{1/8}}.
\end{split}
\end{equation}

	Similar reasoning yields a chain of "larger-or-equal" inequalities which, combined with the one above, finalizes the proof.
	\end{proof}

\begin{lemma}\label{assGarLemma}
	Let $x - \E{x}$ be sub-Gaussian and matrices $\Bj$ have bounded spectrum. Also assume for some positive $\gamma$
	\begin{equation}\label{mucloseboundForAll}
	\max_{j \in J}\infnorm{{\Bj\mu_n(j)}} = O(n^\gamma).
	\end{equation}
	Then for any positive $\u$ on a set of probability at least $1-\exp\brac{-\u^2}$ for $N$ and $n$ going to infinity
	\begin{equation}\label{key}
		\sup_c\abs{\Prob{S(x) \le c} - \Prob{Q(\Zt) \le c}}  = O\brac{\log^{3/2}\abs{J}\frac{\u^{3/4}+\log^{3/8}\abs{J}+n^{3\gamma/4}}{n^{1/8}}},
	\end{equation}
	where $Q(\cdot)$ is defined by \eqref{Qdef}.
\end{lemma}
\begin{proof}
	Application of \ref{LboundLemma} to matrices $E(j)$ yields the bound on $\L$ defined by \eqref{Ldef} $$\L = O(\t + n^\gamma)$$
	 on a set of probability at least $1-\abs{J}e^{-\t^2}$.

	  Investigation of $\s$ defined in \ref{anticoncentration} yields $\sqrt{n}/\s = O(1)$. Really, $\Fnorm{\Bj}$ is a sum of squared eigenvalues (which are bounded) and $\norm{\E{x}}^2 \le n \infnorm{\E{x}}^2$. Now we apply \ref{garlemma}
	\begin{equation}\label{key}
	\begin{split}
	\sup_c&\abs{\Prob{S(x) \le c} - \Prob{Q(\Zt) \le c}} \le R_A    =O\brac{\log^{3/2}\abs{J}\frac{\t^{3/4} + n^{3\gamma /4}}{n^{1/8}}}.
 	\end{split}
	\end{equation}
	Now change the variable $\u^2 \coloneqq \t^2-\logJ$
	\begin{equation}\label{Radef}
	R_A =O\brac{\log^{3/2}\abs{J}\frac{\u^{3/4}+\log^{3/8}\abs{J}+n^{3\gamma/4}}{n^{1/8}}}.
	\end{equation}
\end{proof}

\begin{lemma}\label{onetermdeeper}
	In terms of \ref{boundPhi} for function
	\begin{equation}\label{key}
	\zeta \coloneqq \gDelta \circ \phi
	\end{equation}
	it holds that
	\begin{equation}\label{key}
	\begin{split}
	\sum_{j\in I_i} \brac{\partial_{a_j}+\partial_{b_j}}\zeta(a,b) \le 2\Delta^{-1}\varphi',
	\end{split}
	\end{equation}
	\begin{equation}\label{key}
	\begin{split}
	\abs{\sum_{j,j'\in I_i} \brac{\partial^2_{a_ja_{j'}}+\partial^2_{b_jb_{j'}}}\zeta(a,b)} \le  8\Delta^{-2}\varphi'^2 + 2\Delta^{-1}\varphi'',
	\end{split}
	\end{equation}
	and
	\begin{equation}\label{key}
	\begin{split}
	\abs{\sum_{j,j',j''\in I_i} \brac{\partial^3_{a_ja_{j'}a_{j''}}+\partial^3_{b_jb_{j'}b_{j''}}}\zeta(a,b)}& \le \Z \\&\coloneqq  32\Delta^{-3} \varphi'^3 + 24\Delta^{-2}\varphi''\varphi' + 2\Delta^{-1}\varphi'''.
	\end{split}
	\end{equation}
\end{lemma}
\begin{proof}
	The proof consists in direct differentiation and bounding using \ref{boundPhi} and equation (53) from \cite{doi:10.1137/S0040585X97T988411}. Intermediate differentiation steps can be found in the proof of Lemma A.14 \cite{doi:10.1137/S0040585X97T988411}.
\end{proof}
The following lemma justifies the smoothing relying on smooth indicator $\gDelta$ and soft-max $\Fbeta$. Its proof can be found in \cite{chernozhukov2013}.
\begin{lemma}\label{boundsForSmoothing}
	Let $\Delta = \logJ/\beta$, then for arbitrary vector $x$ :
	\begin{equation}\label{key}
	\gDelta(\Fbeta(x-\Delta)) \le 1\left\{\infnorm{x} \ge 0 \right\} \le \gDelta(\Fbeta(x+\Delta)).
	\end{equation}
\end{lemma}

The next lemma establishes prerequisites for inequality \eqref{Lbound}.
\begin{lemma}\label{LboundLemma} Consider a symmetric matrix $A$ with the largest eigenvalue $\Lambda$.
	Let $\eps$ be a vector of independent sub-Gaussian with $\g^2$ elements.
	Then on a set of probability at least $1-\exp(-\t^2)$
	\begin{equation}\label{Lbound}
	\infnorm{A\eps} \le \Lambda \t.
	\end{equation}
\end{lemma}
\begin{proof}
	For a given unit vector $a$, as far as the components of $\eps$ are independent and sub-Gaussian, $a^T\eps$ is sub-Gaussian with $\g^2$ as well. Hence,
	\begin{equation}\label{key}
	\Prob{\abs{a^T\eps} > t} \le p \coloneqq 2\exp(-\t^2/\g^2)
	\end{equation}
	and therefore,
	\begin{equation}\label{key}
	\infnorm{A\eps} \le \norm{A\eps} \le \Lambda \t.
	\end{equation}
\end{proof}

%% file: comparison.tex
\section{Gaussian comparison}
Notation of this section follows the notation of Section \ref{sectionGAR}. Proof of the following result was inspired by the proof of Theorem 1 in \cite{chernComp}.

\begin{lemma}\label{lemmacomparison}
	Consider two $2N$-dimensional normal vectors $Z\sim \N{\mu}{\Sigma}$ and $\tilde Z\sim \N{\tilde\mu}{\tilde\Sigma}$. Denote $\Delta_\mu \coloneqq \infnorm{\mu - \tilde{\mu}}$ and $\Delta_{\Sigma} \coloneqq \infnorm{\Sigma - \tilde \Sigma}$. Use notation of \ref{boundPhi}. Then for any constant vector $c$  and positive $\Delta$  holds
	\begin{equation}\label{key}
	\begin{split}
	\abs{\Prob{Q(Z) < c} - \Prob{Q(\Zt) < c}} \le \Delta_\mu\brac{\frac{\L\sqrt{n}}{\Delta  }} + 16\Delta_\Sigma  \brac{\frac{\L+\logJ\L^2}{\Delta^2}  } \\+ \frac{8\sqrt{n}}{\s}\Delta\sqrt{4\logJ},
	\end{split}
	\end{equation}
	where $\s$ comes from \ref{anticoncentration}.
\end{lemma}
\begin{proof}
	The proof consists in a multiple use of Slepian smart interpolant.
	Denote the first and the second halves of vector $Z$ as $x$ and $X$ and similarly introduce $\xtt $ and $\tilde X $ being halves of $\Zt$. Further, consider $n$ real values $\varphi_1, \varphi_2, .., \varphi_n$ and compose a vector of length $N$ iterating over these values:
	\begin{equation}\label{key}
	\psi(\varphi) \coloneqq \psi\brac{\{\varphi_i\}_{i=1}^n} \coloneqq \brac{\varphi_1, \varphi_2, .., \varphi_n, \varphi_1, \varphi_2... }.
	\end{equation}
	Denote $f \coloneqq \gDelta \circ \Fbeta\circ Q$ and consider a function
	\begin{equation}\label{key}
	\Psi(\psi) \coloneqq \E{f(\underbrace{x\otimes \sqrt{\psi} + \xtt \otimes \sqrt{1-\psi}}_{x(\psi)}, \underbrace{X\otimes \sqrt{\psi} + \tilde X  \otimes \sqrt{1-\psi}}_{X(\psi)})},
	\end{equation}
	where we use $\otimes$ to denote element-wise product and  radicals are applied to vectors in an element-wise manner. Clearly,
	\begin{equation}\label{key}
	\Psi(1) = \E{f(x, X)} \text{ and } \Psi\brac{0} = \E{f(\xtt, \tilde X )}
	\end{equation}
	and hence
	\begin{equation}\label{integralRepresentation}
	\begin{split}
	\abs{\E{f  (x, X)}\right. - \left.\E{f (\xtt, \tilde X )}} = \abs{\int\limits_0^{1} ... \int\limits_0^{1} {\brac{\prod_{i=1}^{n}\partial_{\varphi_i}} \Psi(\psi(\varphi))}\prod_{i=1}^{n} d\varphi_i}.
	\end{split}
	\end{equation}
	For the derivative we have
	\begin{equation}\label{key}
	\begin{split}
	\brac{\prod_{i=1}^{n}\partial_{\varphi_i}}\Psi(\psi(\varphi)) = \frac{1}{2}\E{\sum_{i=1}^n\sum_{j=I_i} \partial_{x_j} f(x(\psi), X(\psi))\brac{x_j \psi_i^{-1/2} - \xtt_j (1-\psi_i)^{-1/2}}  \right. \\+ \left.  \partial_{X_j} f(x(\psi), X(\psi))\brac{X_j \psi_i^{-1/2} - \tilde X _j (1-\psi_i)^{-1/2}}}.
	\end{split}
	\end{equation}
	Next we apply \ref{stein} (which applies only to centered vectors, thus the second term)
	\begin{equation}\label{key}
	\begin{split}
	\abs{\brac{\prod_{i=1}^{n}\partial_{\varphi_i}}\Psi(\psi(\varphi))} \le  \frac{1}{2}&\Delta_\Sigma\abs{\E{\sum_{i=1}^n\sum_{j=I_i}\sum_{k=I_i} \brac{\partial_j^2+\partial_k^2} f(x(\psi), X(\psi)) }}\\
	&+\frac{1}{2}\Delta_\mu\abs{\sum_{i=1}^n\sum_{j=I_i} \partial_{x_j} f(x(\psi), X(\psi))}  \\\eqqcolon \frac{1}{2}&\brac{T_1+T_2}.
	\end{split}
	\end{equation}
	Now we make use of \ref{onetermdeeper} and \ref{boundPhi} and choose $\beta=\logJ / \Delta$
	\begin{equation}\label{key}
	T_1 \le 32\Delta_\Sigma  \brac{\frac{\L+\logJ\L^2}{\Delta^2}  },
	\end{equation}
	\begin{equation}\label{key}
	T_2 \le \Delta_\mu\brac{\frac{2\L \sqrt{n}}{\Delta }}.
	\end{equation}
	Next, recalling \eqref{integralRepresentation} obtain
	\begin{equation}\label{key}
	\abs{\E{f  (x, X)} - \E{f (\xtt, \tilde X )}} \le \Delta_\mu\brac{\frac{\L\sqrt{n}}{\Delta  }} + 16\Delta_\Sigma  \brac{\frac{\L+\logJ\L^2}{\Delta^2}  }.
	\end{equation}
	Finally, in order to move from smooth functions to indicators we employ reasoning identical to the one in \ref{garlemma}.
	\begin{equation}\label{key}
	\begin{split}
	\Prob{Q(Z) < c} & \le  \E{\gDelta \circ \Fbeta \brac{Q(x, X) - c - \Delta}} \\
	& \le \E{\gDelta \circ \Fbeta \brac{Q(\xtt, \tilde X ) - c - \Delta}} +  \frac{1}{2}\brac{T_1+T_2}\\
	& \le \Prob{Q(\Zt) < c -2 \Delta} + \frac{1}{2}\brac{T_1+T_2}\\
	& \le  \Prob{Q(\Zt) < c } + \frac{1}{2}\brac{T_1+T_2} + \frac{8\sqrt{n}}{\s}\Delta\sqrt{4\logJ}.
	\end{split}
	\end{equation}
	Combination with a similar chain of larger-or-equal  finalizes the proof.
\end{proof}
We use the same version of Stein's identity as the authors of \cite{chernComp} have.
\begin{lemma}[Stein’s identity]\label{stein}
	Let $X \in \R^p$ be a normal centered vector and function $f:\R^p \rightarrow \R$ be a $C^1$ function with finite first derivatives. Then for all $j=1..p$
	\begin{equation}\label{key}
	\E{X_jf(X)} = \sum_{k=1}^{p}\E{X_jX_k}\E{\partial_k f(X)}.
	\end{equation}
\end{lemma}
\begin{proof}
	See Section A.6 of \cite{Bovier:2004:SGC:1030154.1030162}.
\end{proof}

\begin{lemma}\label{sigmabootisclose}
	Consider $y, \yhat, \eps,\epsboot$ defined in Section \ref{sectionmethod}. Let $\Delta_f \coloneqq \infnorm{\E{y}-\yhat} = O\brac{(\log \Ibsize / \Ibsize)^{\kappa}}$ for some positive $\kappa\le 1/2$. Let $\eps_i$ be sub-Gaussian with $\g^2$. Then on a set of probability at least $1-\Ibsize^{10}$
	\begin{equation}\label{key}
	\abs{\Var{\eps_1} - \Varb{\epsboot_1}} = O\brac{\Delta_f^2}.
	\end{equation}
\end{lemma}
\begin{proof}
	By construction
	\begin{equation}\begin{split}
	\Varb{\epsboot_1} &= \frac{1}{\Ibsize}\sum_{i \in \Iboot}  \brac{\hat\eps_i - \underbrace{\frac{1}{\Ibsize}\sum_{i \in \Iboot} \hat\eps_i}_{\bar{\eps}}}^2 \\
	&=\frac{1}{\Ibsize}\sum_{i \in \Iboot} \brac{\E{y_i}-\yhat_i + \eps_i - \bar\eps}^2 \\
	&= \frac{1}{\Ibsize}\sum_{i \in \Iboot} \brac{\E{y_i}-\yhat_i}^2 + \eps_i^2 + \bar\eps^2  \\ &~~~~+ 2\eps_i\brac{\E{y_i} - \yhat_i}-2\bar\eps\brac{\E{y_i} - \yhat_i} - 2\bar\eps\eps_i.
	\end{split}
	\end{equation}
	Now due to sub-Gaussianity for a positive $\e$
	\begin{equation}\label{key}
	\forall i : \Prob{\abs{\eps_i} > \e} \le 2\exp(-\e^2/\g^2)
	\end{equation}
	and hence
	\begin{equation}\label{key}
	\Prob{\underbrace{\exists i : \abs{\eps_i} > \e}_\mathcal{E}} \le p' \coloneqq  2\Ibsize\exp(-\e^2/\g^2).
	\end{equation}
	On set $\mathcal{E}$ Hoeffding inequality applies to $\eps_i$ and their squares:
	\begin{equation}\label{key}
	\Prob{\abs{\frac{1}{\Ibsize}\sum_{i \in \Iboot}\eps_i} > E} \le p''\coloneqq 2\exp\brac{-2\Ibsize E^2/\e},
	\end{equation}
	\begin{equation}\label{key}
	\Prob{\abs{\frac{1}{\Ibsize}\sum_{i \in \Iboot}\eps_i^2 - \Var{\eps_i}} > E'} \le p'''\coloneqq 2\exp\brac{-2\Ibsize E'^2/\e^2}.
	\end{equation}
	Therefore,	 with probability at least $1-p'-p''-p'''$
	\begin{equation}
		\abs{\bar\eps}\le \Delta_f + E
	\end{equation}
	and hence
	\begin{equation}\label{key}
	\abs{\Var{\eps_i} - \Varb{\epsboot_i} }\le \Delta_f^2 + E' + 2E\Delta_f + 2\brac{\Delta_f + E} (\Delta_f+2E+E^2).
	\end{equation}
	Clearly, the choice
	\begin{equation}\label{key}
	\e = \sqrt{\log 2\Ibsize},
	\end{equation}
	\begin{equation}\label{key}
	E = \sqrt{\log (2\Ibsize)/\Ibsize},
	\end{equation}
	\begin{equation}\label{key}
	E' = {\log \brac{2\Ibsize} \bigg/ \sqrt{\Ibsize}}.
	\end{equation}
	makes  $p',p'',p'''$ polynomially decreasing.
	Substitution yields the claim.
\end{proof}

%% file: sensitivityProof.tex
\section{Proof of sensitivity result} \label{sensProofsec}
\begin{proof}[Proof of \ref{sensitivityTheorem}]
	Denoting the probability density functions in the world of the Gaussian Process Regression model as $p(\cdot)$ by construction we have
	\begin{equation}\label{key}
	\begin{split}
	A_n(t) &= \log \frac{p(\ylt) p(\yrt)}{p(\yt)} \\
	&=\log \frac{p(\yrt)}{p(\yrt|\ylt)}.
	\end{split}
	\end{equation}
	Further, denote $\frt \coloneqq \E{\yrt}$ and $\epsrt \coloneqq \yrt - \frt$. Define shorthand notation $\Krt\coloneqq K(\Xrt)$ and $\Klt \coloneqq K(\Xlt)$. Also let $\hatfrt$ and $\hatV$ denote predictive mean and variance of the Gaussian Process Regression for $\Xrt$ given $\Xlt$ and $\ylt$. Now recall the posterior is Gaussian:
	\begin{equation}
		p(\yrt|\ylt) = \N{\yrt|\hatfrt}{\hatV}.
	\end{equation}
	Define a norm $\norma{x}{A} \coloneqq \norm{A^{1/2}x}$ for an arbitrary positive-definite symmetric matrix $A$. Clearly, $\normasq{x}{A} = \dotprod{x}{Ax}$. Now trivial algebra yields

	\begin{equation}
	\begin{split}
		\Ant \cong &\underbrace{\normasq{\frt - \hatfrt}{\invhatV}}_{T_1} +  \underbrace{\normasq{\epsrt}{\invhatV - \invKrt}}_{T_2} \\&+  \underbrace{2\dotprod{\brac{\frt-\hatfrt}\invhatV - \frt\invKrt}{ \epsrt}}_{T_3} ,
	\end{split}
	\end{equation}
	where we use $\cong$ to denote ``equality up to an additive deterministic constant''.
	Consider a matrix $K(\Xt)$ being a block $2\times 2$ matrix with blocks of equal size:
	\begin{equation}
		K(\Xt) = \left(\begin{array}{cc}
				\Klt& \K \\
				\K^T & 		\Krt
		\end{array}
		\right).
	\end{equation}
	Notice that $\hatV$ is its Schur complement, thus $\eigmax{\hatV} \le \eigmax{K(\Xt)} \le C$ (the second inequality is due to \ref{assboundedEigenvalues}).
	Using $\sigma^2>0$ we have $\eigmin{\hatV} > 1/c$ and $\eigmin{\Krt} > 1/c$ for some $c$ independent of $n$. To sum these observations up:
	\begin{equation}
		\exists c > 0: \eigmin{\hatV}, \eigmax{\hatV}, \eigmax{\invKrt} \in (1/c, c).
	\end{equation}
	Having established control over these eigenvalues, we are ready to bound the terms $T_2$ and $T_3$ from above under both $\Hnull$ and $\Hone$, while $T_1$ should be bounded from above under $\Hnull$ and from below under $\Hone$.
	Now we bound the test statistic $\Ant$ under $\Hnull$. Denote $\Deltaf \coloneqq \infnorm{\frt - \hatfrt}$. Then
	\begin{equation}
		T_1 \le c n \Deltaf^2.
	\end{equation}
	In order to bound the second term on a set of high probability we employ \ref{quadtails} and obtain for a positive $\t$
	\begin{equation}\label{T2bound}
		\Prob{\abs{T_2 - \E{T_2}}  \ge {4\g^2c\sqrt{n\t}}} \le \exp(-\t).
	\end{equation}
	The third term will be controlled using sub-Gaussianity of $\epsrt$. For any unit vector $u$ and positive $\e$
	\begin{equation}
	\Prob{ \abs{\dotprod{u}{\epsrt}} \ge \g\sqrt{\e}} \le  2\exp(-\e)
	\end{equation}
	and clearly,
	\begin{equation}
		\norm{\brac{\frt-\hatfrt}\invhatV - \frt\invKrt} \le 4\sqrt{n}c(F + \Deltaf),
	\end{equation}
	where $F \coloneqq \max\{\infnorm{f}, \infnorm{\fone}, \infnorm{\ftwo}\}$. Hence, on a set of probability at least $1-2\exp(-\e)$
	\begin{equation}\label{T3bound}
		\abs{T_3} \le 8\sqrt{n}\g\e c(F + \Deltaf).
	\end{equation}
	Finally, we choose $\t \coloneqq 10\log n$ and $\e \coloneqq 10\log n$. Now under $\Hnull$ bound $\Deltaf$ by \ref{gpconsistency}, recall $\kappa < 1/2$ and obtain
	\begin{equation}
	\begin{split}
		\abs{\Ant - \E{T_2}} &= O\brac{n^{1-2\kappa}\log^{2\kappa}n +  \sqrt{n\log n}} \\
		&= O\brac{n^{1-2\kappa}\log^{2\kappa}n}
	\end{split}
	\end{equation}
	on a set of probability at least $1-3/n^{10} \rightarrow 1$ as $n\rightarrow+\infty$.
	Now we use \ref{validityTheorem} along with the fact that for $n$ large enough $\alpha > R_A + 3/n^{10}$ and obtain on a set of probability approaching $1$
	\begin{equation}
	\abs{\xntboot -  \E{T_2}} = O\brac{n^{1-2\kappa}\log^{2\kappa}n}.
	\end{equation}
	On the other hand, under $\Hone$ the bounds \eqref{T2bound} and \eqref{T3bound} still hold and
	\begin{equation}
		\begin{split}
			T_1 &\ge \eigmin{\hatV} \norm{\frt - \hatfrt}^2 \\
			&\ge  \frac{1}{c}n\brac{\B_n^2-\Delta_f^2}.
		\end{split}
	\end{equation}
	Finally, choose $n=\nstar$, $t = \tau$, and recall assumption \eqref{deltaRate} to conclude $A_n(\tau) > x^\flat_{n,\alpha}(\tau)$  for large $n$ with probability approaching $1$.
\end{proof}


The following result bounds a quadratic form of a sub-Gaussian vector with high probability. It is a direct corollary of Theorem 1.1 (Hanson-Wright inequality) stated in \cite{rudelson2013}.
\begin{lemma}\label{quadtails}
	Consider a vector $x\in \R^n$ sub-Gaussian with $\g^2$ and a positive-definite matrix $A$ of size $n\times n$. Let there be a constant $\Lambda$ independent of $n$ s.t. $\eigmax{A} \le \Lambda$. Then for a positive $\t$, large enough $n$ and some absolute positive $c$
	\begin{equation}
		\Prob{\abs{\dotprod{x}{xA} -\E{\dotprod{x}{xA}}}\ge c\g^2 \Lambda\sqrt{n\t}} \le \exp(-\t).
	\end{equation}
\end{lemma}

%% file: anticoncentration.tex
\section{Anti-concentration inequality}
This section uses notation introduced in Section \ref{sectionGAR}.
\begin{lemma}\label{anticoncentration}
	Consider a $2p$-dimensional Gaussian vector $z = (x, X)$, where $x$ and $X$ are $p$-dimensional. Further, let $\Var{x} = \sigma^2 I_p$ and $Cov(x_j, X_j) = Cov(\eta, \eta^2)$ for arbitrary $1 \le j \le p$ and $\eta \sim \N{0}{\sigma^2}$. Finally, let $\Var{X} = \Var{\eta^2}I_p$.  Then
	for an arbitrary vector $C$ and $\delta \in \R$
	\begin{equation}
		\Prob{Q( z) < C + \delta} - \Prob{Q( z) < C} \le 3\sqrt{\frac{n}{\s^2}}\delta\left(\sqrt{4\log p} + 2\right)
	\end{equation}
	where
	\begin{equation}
		\s^2 \coloneqq \brac{\min_{j\in J}\Fnorm{\Bj}^2+\norm{\E{x}}^2}\sigma^2
	\end{equation}
	 and the map $Q(\cdot)$ is defined by \eqref{Qdef}.
\end{lemma}
\begin{proof}
	Introduce an isotropic Gaussian vector $\tilde z = \Var{z}^{-1/2}z\Var{z}^{-1/2}$ and notice, $\sigma^2/3 \le \Var{\tilde z_j} \le 3\sigma^2$ for all $j$. Applying \ref{anticoncentrationKEY} to $\tilde z$ yields the claim.
\end{proof}

The rest of the proofs of this section mostly follow the Nazarov's inequality proof presented in \cite{nazarov}.

Define a map $u : \R^{2p}\rightarrow \U$, where $\U \coloneqq  \R^{(p+5)p/2}$:
\begin{equation}\label{key}
	\begin{split}
		&u\left((x_1, x_2, .. , x_p)^T, (X_1, X_2, .. , X_p)^T\right) \\ = &(x_1, x_2, .. , x_p, X_1, X_2, .. , X_p,x_1^2, x_2^2, .. , x_p^2, \sqrt{2}x_1x_2, \sqrt{2}x_1x_3, .., \sqrt{2}x_{p-1}x_p)^T.
	\end{split}
\end{equation}
 With a slight abuse of notation we will use $u$ to denote both the map and an element of its image.
\begin{lemma}\label{anticoncentrationKEY}
	Consider $x \sim \N{0}{I_{2p}}$ and $a_1, a_2, .. a_{p(p+5)/2} \in \{u \in \U : \norm{u} = 1\}$ along with $b_1, b_2,.., b_{p(p+5)/2} \in \R$. Then for all positive $\delta$:
	\begin{equation}\label{key}
	\Prob{\forall j : a_j^Tu(x) \le b_j + \delta} - \Prob{\forall j : a_j^Tu(x) \le b_j} \le \delta\left(\sqrt{4\log p} + 2\right).
	\end{equation}
\end{lemma}
\begin{proof}
	Define a  set $K(t) \coloneqq \left\{u \in \U:\forall j~ a_j^Tu \le b_j + t\right\}$, and a function $G(t) \coloneqq \Prob{u(x) \in K(t)}$. $G$ is absolutely continuous distribution function and hence
	\begin{equation}\label{key}
	G(\delta) - G(0) = \int_{0}^{\delta} G'_+(t)d t,
	\end{equation}
	where $G'_+$ denotes the right derivative of $G$.
	Essentially, the proof boils down to the following lemma.
	\begin{lemma}
		\begin{equation}\label{key}
		\lim_{\delta\downarrow 0}\frac{\Prob{u(x) \in K(\delta) \backslash K(0)}}{\delta} \le \delta\left(\sqrt{4\log p} + 2\right).
		\end{equation}
	\end{lemma}
	\begin{proof}
		Denote $K \coloneqq K(0)$ and note it is a convex polyhedron. Denote a projector onto $K$ as $P_K$: $\norm{x - P_Kx} = \min_{y\in K}\norm{x-y}$. Now for a (proper) face $F$ of $K$ define
		\begin{equation}\label{key}
		N_F = \left\{u\in \U \backslash K : P_Ku \in \relint(F) \right\},
		\end{equation}
		\begin{equation}\label{key}
		N_F(\delta) = N_F\bigcap K(\delta).
		\end{equation}
		Clearly, $K(\delta) \backslash K = \bigcup_{F : \text{face of }K} N_F(\delta)$. Clearly,
		\begin{equation}\label{key}
		\exists C > 0 :  N_F\subset \left\{u \in \U  : \dist(u, F) \le C\delta \right\},
		\end{equation}
		hence for any face $F$ of dimensionality less than $\dim \U - 1$
		for $\delta \downarrow 0: \gamma_p(N_F) \coloneqq  \Prob{u(x) \in N_F} = o(\delta)$. Hence,
		\begin{equation}\label{key}
		\gammap(K(\delta) \backslash K)  =  \gammap \brac{\bigcup_{F : \text{facet of }K} \NFdelta} + o(\delta) \text{ as } \delta \downarrow 0.
		\end{equation}
		Now it is left to prove that
		\begin{equation}\label{key}
		\sum_{F :\text{facet of }K } \lim_{\delta\downarrow 0} \frac{\gammap(\NFdelta)}{\delta} \le \sqrt{4 \log p} + 2.
		\end{equation}
		By \ref{lemma3}
		\begin{equation}\label{key}
		\lim_{\delta\downarrow 0} \frac{\gammap(K(\delta) \backslash K)}{\delta} = \sum_{F : \text{facet of }K}  \int_{u(x) \in F} \fp (x) d\sigma(u(x))),
		\end{equation}
		where $\fp(\cdot)$ denotes the density of $\N{0}{I_{2p}}$.
		Consider facets $F$ such that $\dist(0, F) > 4 \log p$. Choose $\barh = \dist(0,F)v$ (or flip the sign if $\barh \notin F$) and denote $\bar x = \uinv(\barh)$. Further, since $\norm{\bar x} \ge \sqrt{4\log p}$,
		\begin{equation}\label{key}
		\fp(x) = f_1(\norm{\bar x}) f_{2p-1}(k)\le \frac{f_{2p-1}(k)}{p^2}
		\end{equation}
		and given the number of facets is less than $p^2$,
		\begin{equation}\label{facetsFar}
		\sum_{\substack{F : \text{facet of }K \\ \dist(0,F)>4\log p}}  \int_{u(x) \in F} \fp (x) d\sigma(u(x))) \le 1.
		\end{equation}
		Now turn to the facets $F$ s.t. $\dist(0, F) \le 4 \log p$.
		By \ref{lemma3},
		\begin{equation}\label{key}
		\int_{u(x) \in F} \fp (x) d\sigma(u(x))) \le  \brac{\sqrt{4\log p} + 1}\gammap(N_F).
		\end{equation}
		The final observation is based on the fact that $N_F$ are disjoint and $\gamma_p(\U) = 1$
		\begin{equation}\label{facetsClose}
		\sum_{\substack{F : \text{facet of }K \\ \dist(0,F) \le 4\log p}}  \int_{u(x) \in F} \fp (x) d\sigma(u(x))) \le  {\sqrt{4\log p} + 1}
		\end{equation}
		and its combination with \eqref{facetsFar} completes the proof.
	\end{proof}

\end{proof}

\begin{lemma}\label{lemma3}
	\begin{equation}\label{key}
	\lim_{\delta\downarrow 0} \frac{\gammap(\NFdelta)}{\delta} = \int_{u(x) \in F} \fp (x) d\sigma(u(x))) \le \brac{\sqrt{\dist(0,F)} + 1}\gammap(N_F),
	\end{equation}
	where $d\sigma$ is the standard surface measure on $F$.
\end{lemma}
\begin{proof}
	Parametrize every $h \in F$ as
	\begin{equation}\label{key}
	h= h(\barh, z) =  \barh + z_1q_1 + z_2q_2 + ...+ z_{p(p+5)/2-1}q_{p(p+5)/2-1},
	\end{equation}
	where $\barh$ is an arbitrary element of $F$, while $q_j$ form an orthonormal basis on $F-\barh$. Further, choose a unit outward normal vector $v$ to $\partial K$ at $F$.  Then we can parametrize $N_F$
	\begin{equation}\label{key}
	y = y(\barh, z, t) = h(\barh, z) + tv \text{ for } t > 0.
	\end{equation}
	Now
	\begin{equation}\label{key2}
	\begin{split}
	\gammap(N_F) &= \int_{y=y(\barh, z, t ) \in N_F} \fp\left(\uinv(y)\right)dzdt\\
	& = \int_{h = h(\barh, z) \in F}\left( \int_{0}^{+\infty} \fp\left(\uinv(y)\right)dt  \right) dz
	\end{split}
	\end{equation}
	and in the same way
	\begin{equation}\label{key}
	\gammap(\NFdelta) = \int_{0}^{\delta} \left(\int_{h = h(\barh, z) \in F} \fp\left(\uinv(y)\right)dz \right)dt.
	\end{equation}
	Thus
	\begin{equation}\label{key}
	\lim_{\delta\downarrow 0} \frac{\gammap(\NFdelta)}{\delta} = \int_{h = h(\barh, z) \in F} \fp\left(\uinv(y)\right)dz,
	\end{equation}
	which proves the equality in the claim.

	Now for any $h \in F$ and $v$ exist vectors $x$ and $n$ such that $\uinv(h+tv) = x + t'n$  for  $t' = \sqrt{t}$. Then
	\begin{equation}\label{key3}
	\begin{split}
	\int_{0}^{+\infty} &\fp\brac{\uinv(h+tv)}dt = \int_{0}^{+\infty} \fp\brac{x+t'n}dt' \\
	& \ge \fp(x) \int_{0}^{+\infty} e^{-t |x^Tn|}e^{-t^2/2}dt
	\ge \frac{\fp(x)}{|x^Tn|+1}.
	\end{split}
	\end{equation}
	Combination of \eqref{key2} and \eqref{key3} yields
	\begin{equation}\label{key}
	\gammap(N_F) \ge \int_{h\in F} \frac{1}{|x^Tn|+1} \fp(x)dz, \text{ where } x = \uinv(h).
	\end{equation}
	Now choose $h=\dist(0,F)$, note $|x^Tn| = \sqrt{\dist(0,F)}$ and establish the claim.
\end{proof}

%% file: freqGP.tex
\section{Consistency of Gaussian Process Regression by \cite{Yang2017}}
In this section we quote a consistency result for predictions of Gaussian Process Regression.
\begin{lemma}[Corollary 2.1 in \cite{Yang2017}]\label{gpconsistency} Assume, $\eps_i$ are sub-Gaussian. Let $\fstar$ be $\aleph$-smooth Sobolev and $\kappa \coloneqq \brac{\aleph-1/2}/(2\aleph)$  or $\aleph$-smooth Hölder and $\kappa\coloneqq \aleph/(2\aleph+1)$. Further let $k(\cdot, \cdot)$ satisfy \ref{keras1} and \ref{keras2}. Then, for the training sample size $n$ going to infinity with probability at least $1-n^{-10}$ we have
    \begin{equation}
    \infnorm{\fstar - \hat f} = O\brac{\brac{\frac{\log n}{n}}^{\kappa}},
    \end{equation}
    where $\hat f$ denotes the predictive function.
\end{lemma}